\documentclass[a4paper,11pt]{amsart}
\usepackage{amsmath,amssymb,amsthm}
\usepackage{fullpage}
\usepackage{tikz}
\usepackage{graphicx}

\newtheorem{theorem}{Theorem}[section]
\newtheorem{lemma}[theorem]{Lemma}
\newtheorem{proposition}[theorem]{Proposition}
\newtheorem{corollary}[theorem]{Corollary}

\newtheorem{example}[theorem]{Example}

\theoremstyle{definition}
\newtheorem{definition}[theorem]{Definition}

\newcommand{\G}{\ensuremath{\mathcal{G}}}
\newcommand{\D}{\ensuremath{\mathcal{D}}}

\begin{document}

\title{Existence and Non-existence Results for Strong External Difference Families}

\author{Sophie Huczynska}
\address{School of Mathematics and Statistics, University of St Andrews, St Andrews, Scotland, U.K.}
\email{sh70@st-andrews.ac.uk}

\author{Maura B.~Paterson}
\address{Department of Economics, Mathematics and Statistics, Birkbeck University of London, London, U.K.}
\email{m.paterson@bbk.ac.uk}

%\author{S. Huczynska and M.B.~Paterson}
\bibliographystyle{abbrv}
%\begin{document}

\begin{abstract}
We consider strong external difference families (SEDFs);  these are external difference families satisfying additional conditions on the patterns of external differences that occur, and were first defined in the context of classifying optimal strong algebraic manipulation detection codes.  We establish new necessary conditions for the existence of $(n,m,k,\lambda)$-SEDFs; in particular giving a near-complete treatment of the $\lambda=2$ case.  For the case $m=2$, we obtain a structural characterization for partition type SEDFs (of maximum possible $k$ and $\lambda$), showing that these correspond to Paley partial difference sets.  We also prove a version of our main result for generalized SEDFs, establishing non-trivial necessary conditions for their existence.
\end{abstract}

\maketitle
\section{Introduction}
Difference families are much-studied objects in combinatorial literature, and have been used to construct a range of combinatorial objects, including designs and strongly-regular graphs.  They have also been applied in a variety of settings to provide a natural way of expressing various desirable properties of codes and sequences.

Given an additive abelian group $\G$, a set of disjoint subsets of $\G$ forms a \emph{disjoint difference family (DDF)}, where the differences between pairs of subset elements are called \emph{external differences} if the elements lie in different subsets, and \emph{internal differences} if the elements lie in the same subset.  Additional properties may be imposed: for example all subsets in the family may be of the same size, the subsets may partition the group (sometimes, the non-zero elements of the group) or every non-zero element of $\G$ may arise as a (internal/external) difference from the subsets of the family a constant number of times.  A survey of the area is given in \cite{NgPat}. Historically, the external differences have been somewhat less studied than their internal counterparts.  \emph{External difference families (EDFs)} were introduced in \cite{OgaKurStiSai} to construct optimal secret sharing schemes secure against cheating in the setting where the secrets are uniformly distributed.  They are a special case of both {\em difference systems of sets} \cite{ChaDin}, and (weak) {\em algebraic manipulation detection (AMD) codes} \cite{CraDodFehPadWic}.  AMD codes generalise certain known techniques for constructing secret sharing schemes secure against cheating,  and it is established in \cite{CraDodFehPadWic} that such a code is equivalent to a type of DDF. In the setting of secret sharing schemes secure against sets of cheating participants who know the secret (sometimes referred to as the CDV assumption \cite{CDV}) it is necessary to use a {\em strong} variant of these codes.

 %this robust secret sharing approach, and can be combined with a cryptographic system providing some form of secrecy to faciliate extra robustness against an adversary who is able to change values in the system.  In its basic form, a source is chosen from a finite set $S$ and encoded via a map $E:S \rightarrow G$ where $\G$ is a suitable abelian group; in order for $E(s)$ to uniquely determine $s$, the sets of encodings must be disjoint for different sources.  The adversary's manipulation is to add a group element $d$ to the encoded value;  (s)he is said to have succeeded if $E(s)+d$ is a valid encoding $E(t)$ of some other source $t$ and hence is decoded to $t$.  This set-up is said to form a (weak) AMD code with parameters $(|S|,|G|,\epsilon)$ if, for any choice of $d \in G$, the success probability of the adversary is bounded above by $\epsilon$. In the strong version, the adversary knows $s$ in advance.  It is established in \cite{CraDodFehPadWic} that such a code, with deterministic encoding, is equivalent to a type of DDF.  Further, in \cite{PatSti}, optimal (%weak and strong) AMD codes are shown to correspond to EDFs and their generalizations. 

In this paper, we consider \emph{strong external difference families (SEDFs)}, introduced in \cite{PatSti}.  These are external difference families satisfying an extra condition, and correspond to the strong set-up in the AMD code situation.  The existence of SEDFs is an active area of current investigation (for example, in \cite{MarSti} and \cite{WeiZha}). In this paper, we establish new necessary conditions for the existence of $(n,m,k,\lambda)$-SEDFs; in particular this gives a near-complete treatment of the $\lambda=2$ case.  For $m=2$, we obtain a structural characterization of partition type SEDFs (which have maximal possible $k$ and $\lambda$), showing that these correspond to Paley partial difference sets.  We also prove a version of our main result for generalized SEDFs, establishing non-trivial necessary conditions for their existence.

 \section{Preliminaries}
 
The following definitions are given in \cite{PatSti}:

\begin{definition}
Let $\G$ be an additive abelian group.  For any disjoint sets $A_1,A_2 \subseteq \G$, define the multiset
\[ \mathcal{D}(A_1,A_2)=\{x-y \, | \, x \in A_1, y \in A_2 \}.\]
\end{definition}
 
 \begin{definition}[External difference family]
 Let $\G$ be an additive abelian group of order $n$.
An {\bf $(n,m,k,\lambda)$-external difference family} 
(or {\bf $(n,m,k,\lambda)$-EDF})
is a set of $m$ disjoint $k$-subsets of $\G$, say $A_1, \dots , A_m$, such that
the following multiset equation holds:
\begin{equation*}
\bigcup_{\{i,j:j\neq i\}}\D(A_i,A_j)=\lambda(\G\setminus\{0\}).
\end{equation*}
\end{definition}

\begin{definition}[Strong external difference family]
Let $\G$ be an additive abelian group of order $n$.
An {\bf $(n,m,k,\lambda)$-strong external difference family} 
(or {\bf $(n,m,k,\lambda)$-SEDF})
is a set of $m$ disjoint $k$-subsets of $\G$, say $A_1, \dots , A_m$, such that
the following multiset equation holds for every $i$, $1 \leq i \leq m$:
\begin{equation*}
\bigcup_{\{j:j\neq i\}}\D(A_i,A_j)=\lambda(\G\setminus\{0\}).
\end{equation*}
\end{definition}
An $(n,m,k,\lambda)$-SEDF is, by definition, an $(n,m,k,m\lambda)$-EDF.

Various constraints on the parameters follow from the definition.  The definition requires $m \geq 2$.  It is immediate that $km \leq n$ and, as in the case of general EDFs, double-counting of the differences yields the necessary condition:
\begin{equation}\label{eqn:basic}
 \lambda(n-1)=k^2(m-1).
\end{equation}

Combining these yields the following lemma, proved in \cite{WeiZha}:
\begin{lemma}\label{basiclemma}
For an  $(n,m,k,\lambda)$-SEDF, either
\begin{itemize}
\item $k=1$ and $\lambda=1$; or
\item $k>1$ and $\lambda<k$.
\end{itemize} 
\end{lemma}
\begin{proof}
Combining the two necessary conditions above yields $\lambda(n-1)=k(km)-k^2 \leq kn - k^2$, which rearranges to $\frac{\lambda}{k} \leq \frac{(n-k)}{(n-1)}$, from which the result follows.
\end{proof}

Note this implies that the $k$-sets $A_i$ in an $(n,m,k,\lambda)$-SEDF $\{A_1, \ldots, A_m\}$ can be pairs of elements only for $\lambda=1$, triples only for $\lambda=1,2$, and so on.
 
In \cite{PatSti}, a full description of possible parameters was obtained for the case $\lambda=1$:

\begin{theorem}[ \cite{PatSti}]
There exists an $(n,m,k,\lambda)$-SEDF if and only if $m=2$ and $n=k^2+1$, or $k=1$ and $m=n$.
\end{theorem}

Constructions were given for both of these cases:
\begin{itemize}
\item Let $\G=(\mathbb{Z}_{k^2+1},+)$, $A_1=\{0,1,\ldots,k-1\}$ and $A_2=\{k,2k, \ldots, k^2\}$.  This is a $(k^2+1,2;k;1)$-SEDF.
\item Let $\G=(\mathbb{Z}_n,+)$ and $A_i=\{i\}$ for $1 \leq i \leq n-1$.  This is an $(n,n;1;1)$-SEDF.
\end{itemize}

Recent work by Martin and Stinson \cite{MarSti}, using character theory, has established various SEDF non-existence results, including the following:

\begin{theorem}\label{MarStiThm1}
Let $\{D_1,\ldots,D_m\}$ form an $(n,m,k,\lambda)$-SEDF.  Then $m \neq 3$ and $m \neq 4$.
\end{theorem}

\begin{theorem}\label{MarStiThm2}
If $\G$ is any group of prime order, and $k>1$ and $m>2$, then $\G$ admits no $\{D_1,\ldots,D_m\}$ which form an $(n,m,k,\lambda)$-SEDF.  
\end{theorem}

\section{New necessary conditions for SEDFs}

In this section, we will prove necessary conditions for the existence of SEDFs with $\lambda \geq 2$. 
\begin{theorem}\label{thm:dne}
Suppose there exists an $(n,m,k,2)$-SEDF with $m\geq 3$ and $k\geq 3$.   Then the following inequality must hold:
\begin{align}
\frac{2(k-1)(m-2)}{k(m-1)}\leq 1. \label{eq:lambda2}
\end{align}

%with $m\geq 6$ and $k\geq 3$.
\end{theorem}
\begin{proof}
Suppose there exists an $(n,m,k,2)$-SEDF with $m\geq 3$ and $k\geq 3$.  We will show that, if \eqref{eq:lambda2} does not hold, then it is possible to find a point $v$ in $A_1$ and two internal differences from $v$ which correspond to a point $v^\prime$ in some $A_i$ for $i\neq 1$ and two external differences from $v^\prime$, and thereby to construct three external differences from $A_1$ that are all equal. 

Fix a point $v$ in $A_1$.  Let $I$ be the set of internal differences from $v$
\begin{equation*}
I=\{v-a \, | \, a\in A_1,\ a\neq v\}\subseteq \G\setminus\{0\}.
\end{equation*}
Then $|I|=k-1$, as $|A_1|=k$.

For $x\in A_i$ with $i\neq 1$ let $E_x$ be the set of external differences from $x$ to elements of $A_j$ for any $j\neq 1$:
\begin{equation*}
E_x=\{x-a \, | \, a\in A_j,\ j\neq 1,i\}.
\end{equation*}
Then $|E_x|=(m-2)k$ for any $x$.

From the definition of an $(n,m,k,2)$-SEDF, each nonzero group elements appears twice in the multiset of external differences from $A_j$ for each $j=1,2,\dotsc, m$, and hence $2m$ times in the multiset of all external differences
\begin{equation}
\{a-b \, | \, a\in A_i,\ b\in A_j,\ j\neq i\}=2m({\G}\setminus\{0\}). \label{eq:2externaldifferences}
\end{equation}  

Furthermore, since the multiset of external differences from $A_1$ comprises two copies of each nonzero group element, it is also the case that each nonzero group element occurs precisely twice as an external difference from some $A_i$ for $i\neq 1$ into $A_1$ (since the multiset of such external differences can be obtained by negating the multiset of external differences out of $A_1$). From this we can deduce that each nonzero group element occurs $2(m-2)$ times as an external difference between sets $A_i$ and $A_j$ with $i\neq j$ and $i,j\neq 1$, so
\begin{equation}
\bigcup_{x \in A_i, i\neq 1} E_x =2(m-2)({\G}\setminus \{0\}). \label{eq:2exunion}
\end{equation}

Suppose we could find an element $v^\prime \in A_i$ for some $i\neq 1$ for which $|E_{v^\prime} \cap I|\geq 2$.

\begin{figure}
\begin{center}
\begin{tikzpicture}[scale=1]
%\draw[step=1cm,gray,very thin] (0,0) grid (10,10);

\draw (2,5) ellipse (2cm and 1cm);
\draw (11,4.5) ellipse (2cm and 1cm);
\draw (10,2.5) ellipse (2cm and 1cm);
\draw (7,4) ellipse (2cm and 1cm);

\node at (.8,5) {$A_1$};
\node at (11,4.5) {$A_\ell$};
\node at (10,2.5) {$A_j$};
\node at (7,4) {$A_i$};
%\draw (5,8.5) ellipse (2cm and .5cm);
%\draw (5,6.5) ellipse (2cm and .5cm);
%\draw (5,4.5) ellipse (2cm and .5cm);
%\draw (5,2) ellipse (2cm and .5cm);

\fill[black] (1.5,5)circle (0.05cm);
\fill[black] (2,4.2)circle (0.05cm);
\fill[black] (3,5.4)circle (0.05cm);

\node at (1.5,5.2) {$v$};
\node at (2,3.9) {$u$};
\node at (3,5.6) {$w$};

\draw[arrows=->,line width=1pt](1.5,5)--(2,4.2) node[midway,below] {$\delta_1$};
\draw[arrows=->,line width=1pt](1.5,5)--(3,5.4) node[midway,above] {$\delta_2$};

%\fill[black] (9.5,3.9)circle (0.05cm);
%\fill[black] (7,3.5)circle (0.05cm);
%\fill[black] (6,6)circle (0.05cm);
\fill[black] (8.5,4)circle (0.05cm);
\fill[black] (9,3.2)circle (0.05cm);
\fill[black] (10,4.4)circle (0.05cm);

\node at (8.5,4.3) {$v^\prime$};
\node at (9,2.8) {$u^\prime$};
\node at (10,4.7) {$w^\prime$};

\draw[arrows=->,line width=1pt](8.5,4)--(9,3.2) node[midway,below] {$\delta_1$};
\draw[arrows=->,line width=1pt](8.4,4)--(10,4.4) node[midway,above] {$\delta_2$};

\draw[arrows=->,line width=.5pt](1.5,5)--(8.5,4) node[midway,above] {$\gamma$};
\draw[arrows=->,line width=.5pt](2,4.2)--(9,3.2) node[midway,above] {$\gamma$};
\draw[arrows=->,line width=.5pt](3,5.4)--(10,4.4) node[midway,above] {$\gamma$};
\end{tikzpicture}
\end{center}
\caption{A point $v$ in $A_1$ and two internal differences from $v$, corresponding to a point $v^\prime$ in $A_i$ and two external differences from $v^\prime$, give rise to three equal external differences from $A_1$.}
\label{fig:drawingforlambda2}
\end{figure}
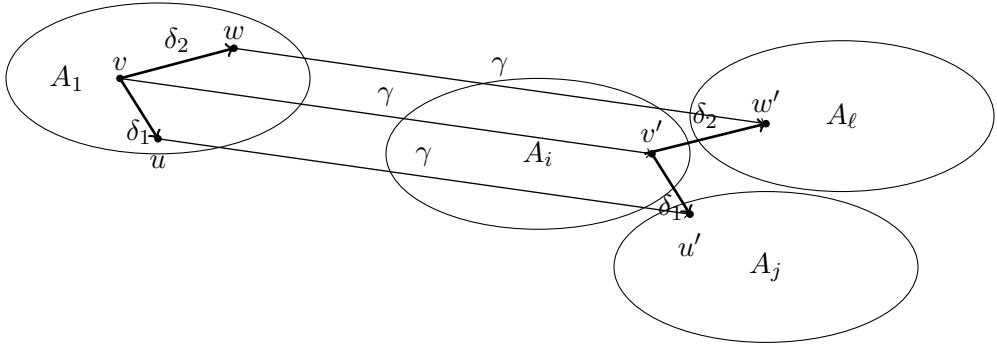

 Let $\delta_1$ and $\delta_2$ be distinct elements of $E_{v^\prime}\cap I$.  Let $u=v-\delta_1$ and $w=v-\delta_2$.  Then $u$ and $w$ are distinct elements of $A_1$, as $\delta_1$ and $\delta_2$ are distinct elements of $I$.  Let $u^\prime=v^\prime-\delta_1$ and $w^\prime=v^\prime-\delta_2$.  Then $u^\prime$ and $w^\prime$ are distinct elements of $\cup_{j\neq 1,i} A_j$ as $\delta_1$ and $\delta_2$ are distinct elements of $E_{v^\prime}$.  (We note that $u^\prime$ and $w^\prime$ may lie in distinct $A_j$ and $A_\ell$, or they may both occur in a single $A_j$ but that does not affect the rest of this argument.)  Let $v-v^\prime=\gamma\in{\G}\setminus\{0\}$. We observe that 
\begin{align*}
u-u^\prime &=(v-\delta_1)-(v^\prime-\delta_1),\\
&=v-v^\prime,\\
&=\gamma, \intertext{and}
w-w^\prime &=(v-\delta_2)-(v^\prime-\delta_2),\\
&=v-v^\prime,\\
&=\gamma.
\end{align*}
This would contradict the assumption that each nonzero group element occurs precisely twice as an external difference from $A_1$.  (This situation is illustrated in Figure \ref{fig:drawingforlambda2}.)  So we have proved that, for every $v^{\prime}$,  $|E_{v^\prime} \cap I|\leq 1$.

We now count the number $N$ of pairs $(\theta,E_x)$ where $\theta\in I\cap E_x$ and $x\in A_i$ for some $i\neq 1$. There are $k-1$ choices for $\theta$.  As each nonzero element of $\G$ occurs $2(m-2)$ times in $\bigcup_{x \in A_i, i\neq 1} E_x $, for each of these $\theta$ there are $2(m-2)$ values of $x$ for which $\theta\in E_x$, so $N=2(k-1)(m-2)$.

The number of distinct sets $E_x$ with $x\in A_i$ for some $i\neq 1$ is $(m-1)k$.  By the Pigeonhole Principle there exists $x$ for which the set $E_x$ contains at least
\begin{equation*}
\frac{N}{(m-1)k}=\frac{2(k-1)(m-2)}{(m-1)k}
\end{equation*}
elements of $I$. If this quantity was strictly greater than one we would have $|E_{v^\prime} \cap I|\geq 2$ for some $v^\prime$.
\end{proof}
Theorem~\ref{thm:dne} eliminates a wide range of values as potential parameters of an SEDF.  Let us consider the regions in which \eqref{eq:lambda2} does not hold.  These are illustrated in Figure~\ref{fig:lambda2}. Theorem~\ref{thm:dne} applies for $m\geq 3$ and $k\geq 3$.  When $m=3$, the left-hand side of \eqref{eq:lambda2} evaluates to $\frac{k-1}{k}$, which is never greater than $1$.  For $m=4$, it becomes $2\frac{2}{3}\frac{k-1}{k},$ which is greater than $1$ whenever $k>4$. For $m>4$, the threshold is achieved when $k\geq 4$.  When $k=3$, the left-hand side of \eqref{eq:lambda2} evaluates to $2\frac{2}{3}\frac{m-2}{m-1}$, which is greater than $1$ whenever $m>5$.  For $k>3$, the value of $1$ is exceeded for $m\geq 5$.

\begin{figure}
\begin{center}
\includegraphics[width=7cm]{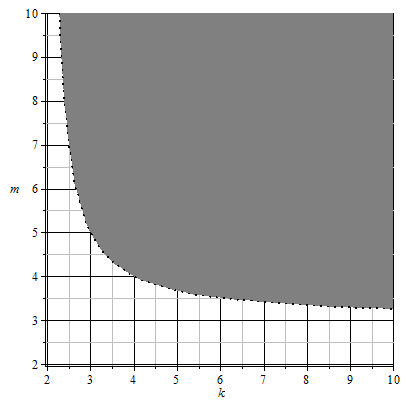}
\end{center}
\caption{A plot depicting values of $m$ and $k$ for which $\frac{2(k-1)(m-2)}{k(m-1)}>1$.  By Theorem~\ref{thm:dne}, if there exists an $(n,m,k,2)$-SEDF then the point $(k,m)$ lies outside the grey region.}
\label{fig:lambda2}
\end{figure}

%\begin{figure}[htb]
%PUT GRAPH HERE
%\caption{asdf}\label{fig:regions}
%\end{figure}

These observations lead directly to the following corollary.

\begin{corollary}\label{cor:lambda=2}
An $(n,m,k,2)$-SEDF can exist only when $m=2$.
\end{corollary}
\begin{proof}
It is immediate from the discussion following Theorem~\ref{thm:dne} that, for an $(n,m,k,2)$-SEDF to exist, its parameters must satisfy one of the following:
\begin{itemize}
\item $k\leq 2$;
\item $m\leq 3$;
\item $k=3$ and $m=4$;
\item $k=3$ and $m=5$;
\item $k=4$ and $m=4$.
\end{itemize}
By Lemma \ref{basiclemma}, we must have $k>2$, so the cases $k=1,2$ cannot occur.  By definition, $m \geq 2$, and the cases with $m=3,4$ cannot occur by Theorem \ref{MarStiThm1}.  The case $k=3$ and $m=5$ corresponds to $n=19$, and hence is ruled-out by Theorem \ref{MarStiThm2}.  Hence only the case when $m=2$ remains.
\end{proof}

In the case when $\lambda=m=2$, equation (\ref{eqn:basic}) shows that $n=\frac{k^2}{2}+1$ (note this implies $n$ and $k$ are coprime). We have the following SEDF with $k=4$ and $n=9$:
\begin{example}\label{lambda=m=2}
Let $\G=(\mathbb{Z}_3 \times \mathbb{Z}_3, +)$, let $A_1=\{(0,1),(0,2),(1,0),(2,0)\}$ and let $A_2=\{(1,1),(1,2),(2,1),(2,2)\}$.  Then $\{ A_1,A_2\}$ is a $(9,2,4,2)$-SEDF.
\end{example}
We shall show later that this example can be viewed as part of a family of SEDFs with $k=\frac{n-1}{2}$ (see Section \ref{sectionPaley}).

The general question ``for which values of $k$ does a $(\frac{k^2}{2}+1,2,k,2)$-SEDF exist?'' remains open.  There are various number-theoretic constraints; for example, prime $k$ are ruled-out by the following result.
\begin{lemma}
An  $(n,2,p,\lambda)$-SEDF, where $p$ is prime, can exist only for $\lambda=1$.
\end{lemma} 
\begin{proof}
Suppose there exists an $(n,2,k,\lambda)$-SEDF where $k=p$, a prime.  By equation (\ref{eqn:basic}), $\lambda(n-1)=p^2$, and $\lambda<p$ by Lemma \ref{basiclemma}.  Since $\lambda$ must divide $p^2$, we must have $\lambda=1$.
\end{proof}

The analogue of Theorem~\ref{thm:dne} in the setting of arbitrary $\lambda$ is:

\begin{theorem}\label{thm:dnelambda}
Let $\lambda\geq 2$.  Suppose there exists an $(n,m,k,\lambda)$-SEDF with $m\geq 3$ and $k\geq \lambda+1$.   Then the following inequality must hold:
\begin{align}
\frac{\lambda(k-1)(m-2)}{(\lambda-1)k(m-1)}\leq 1. \label{eq:lambda}
\end{align}

%with $m\geq 6$ and $k\geq 3$.
\end{theorem}
\begin{proof}
The proof of Theorem~\ref{thm:dne} can readily be adapted to the case of general $\lambda$.  In this setting, if \eqref{eq:lambda} does not hold, then it is possible to find a point $v\in A_1$ and $\lambda$ internal differences from $v$, which correspond to a point $v^\prime$ in some $A_i$ with $i\neq1$ and $\lambda$ external differences from $v^\prime$.  This would then allow the construction of $\lambda+1$ equal external differences from $A_1$ by the same approach we used in the proof of Theorem~\ref{thm:dne}.

The definitions and cardinalities of $I$ and $E_x$ carry over exactly.  Equations \eqref{eq:2externaldifferences} and \eqref{eq:2exunion} become
\begin{align*}
\{a-b \, | \, a\in A_i,\ b\in A_j,\ j\neq i\}&=\lambda m({\G}\setminus\{0\}), \intertext{and}
\bigcup_{x \in A_i, i\neq 1} E_x &=\lambda(m-2)({\G}\setminus \{0\}).
\end{align*}

The situation illustrated in Figure \ref{fig:drawingforlambda2} carries over for general $\lambda$ in the natural way, and our conclusion is that, for every $v^\prime \in A_i$ ($i\neq 1$),  we must have $|E_{v^\prime} \cap I|\leq \lambda-1$.

We count the number $N$ of pairs $(\theta,E_x)$ where $\theta\in I\cap E_x$ and $x\in A_i$ for some $i\neq 1$. There are $k-1$ choices for $\theta$.  Since each nonzero element of $\G$ occurs $\lambda(m-2)$ times in $\bigcup_{x \in A_i, i\neq 1} E_x $ we see that for each of these $\theta$ there are $\lambda(m-2)$ values of $x$ for which $\theta\in E_x$, so $N=\lambda(k-1)(m-2)$.  Applying the Pigeonhole Principle as before, there exists $x$ for which $E_x$ contains at least
\begin{equation*}
\frac{\lambda(k-1)(m-2)}{k(m-1)}
\end{equation*}
 elements of $I$.  If this was strictly greater than $\lambda-1$ we would have $|E_{v^\prime} \cap I|\geq \lambda$ for some $v^{\prime}$.
\end{proof}

Theorem~\ref{thm:dnelambda} applies for $m\geq 3$ and $k\geq \lambda+1$.  When $k=\lambda+1$, the left-hand side of \eqref{eq:lambda} evaluates to $\left(\frac{\lambda^2}{\lambda^2-1}\right)\left(\frac{m-2}{m-1}\right)$, which is greater than one whenever $m>\lambda^2+1$.  For $k>\lambda+1$, we observe that the left-hand side of \eqref{eq:lambda} can be written $\left(\frac{\lambda k -\lambda}{\lambda k-k}\right)\left( \frac{m-2}{m-1}\right)> \left(\frac{\lambda k -k+1}{\lambda k-k}\right)\left( \frac{m-2}{m-1}\right)$, which is greater than one if $m-2>(\lambda-1)k$.

As before, if $m=3$ the left-hand side of \eqref{eq:lambda} cannot be greater than one, as $\frac{\lambda}{\lambda-1}\leq 2$.  When $m=4$, it evaluates to $\left(\frac{2}{3}\frac{\lambda}{\lambda-1}\right)\frac{k-1}{k},$ which is greater than one only in the case where $\lambda=2$ and $k>4$. For larger values of $m$, express the left-hand side as $\left(\frac{(\lambda m -\lambda)-\lambda}{(\lambda m -\lambda) -(m-1)}\right)\left(\frac{k-1}{k}\right)$; we would need $m-1>\lambda$ in order for this to be greater than one.  In this case we have $\left(\frac{(\lambda m -\lambda)-\lambda}{(\lambda m -\lambda) -(m-1)}\right)\left(\frac{k-1}{k}\right)\geq\left(\frac{(\lambda m -\lambda)-(m-1)+1}{(\lambda m -\lambda) -(m-1)}\right)\left(\frac{k-1}{k}\right)$, and for $k-1>(\lambda-1)(m-1)$ the value is greater than one.  The situation is illustrated in Figure~\ref{fig:lambda}.

\begin{figure}
\begin{center}
\includegraphics[width=8.5cm]{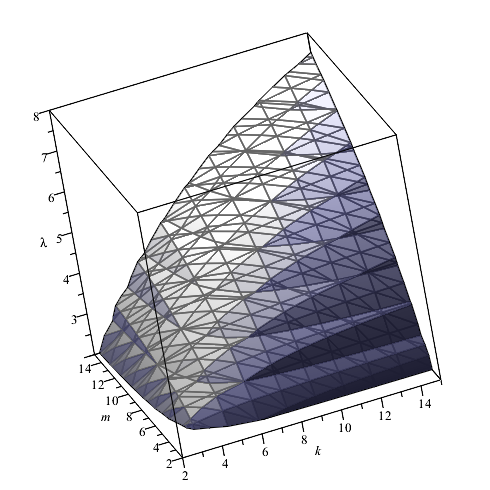}
\end{center}
\caption{A plot depicting the surface consisting of values of $m$, $k$ and $\lambda$ for which $\frac{\lambda(k-1)(m-2)}{(\lambda-1)k(m-1)}=1$.  By Theorem~\ref{thm:dnelambda}, if there exists an $(n,m,k,\lambda)$-SEDF then the point $(k,m,\lambda)$ lies on or above the surface.}
\label{fig:lambda}
\end{figure}

\begin{corollary}\label{cor:lambda_gen}
If there exists an $(n,m,k,\lambda)$-SEDF, then its parameters must satisfy at least one of the following:
\begin{itemize}
\item $\lambda=1$;
\item $m= 2$;
\item $k=\lambda+1$ and $m\leq\lambda^2+1$;
\item $k>\lambda+1$ and $m\leq (\lambda-1)k+2$;
\item $m\geq 5$, $2\leq \lambda\leq m-2$ and $k \leq (\lambda-1)(m-1)+1$.
\end{itemize}
\end{corollary}
\begin{proof}
The following small cases are outside the scope of Theorem \ref{thm:dnelambda} and hence cannot be ruled out by the theorem:
\begin{itemize}
\item $\lambda=1$;
\item $m \leq 2$;
\item $k\leq \lambda$.
\end{itemize}
However, $m \geq 2$ by definition, and by Lemma \ref{basiclemma}, the case $k \leq \lambda$ cannot occur.
By the discussion following the proof of the theorem, we know that any $(n,m,k,\lambda)$-SEDF within the scope of the theorem must have parameters satisfying one (or more) of the following:
\begin{itemize}
\item $k=\lambda+1$ and $m\leq\lambda^2+1$;
\item $k>\lambda+1$ and $m\leq (\lambda-1)k+2$;
\item $m=3$;
\item $m=4$, and either $\lambda>2$ or $k \leq 3$;
\item $m\geq 5$, $2\leq \lambda\leq m-2$ and $k \leq (\lambda-1)(m-1)+1$.
\end{itemize}
The cases with $m=3$ and $m=4$ are ruled out by Theorem \ref{MarStiThm1}.  
\end{proof}

\section{Existence results and characterizations when $m=2$}\label{sectionPaley}

In this section, we consider the $m=2$ case in full generality, i.e. for all $\lambda>1$.  We consider $(n,2,k,\lambda)$-SEDFs with largest possible value of $k$ (and hence $\lambda$).

For number-theoretic reasons, it is not possible to have an SEDF comprising two sets of size $k=\frac{n}{2}$.  In this case, equation (\ref{eqn:basic}) would require $\lambda(2k-1)=k^2$; this cannot happen as $2k-1$ is coprime to $k$. 
The largest possible value of $k$ is therefore $k=\frac{n-1}{2}$ (here $n$ must be odd); this corresponds to the largest possible value of $\lambda=\frac{n-1}{4}$.

Denote by $\G^*$ the non-identity elements of $\G$.  We consider constructions comprising two sets, each of size $\frac{n-1}{2}$, which partition $\G^*$.  The following result, based on a classic cyclotomic construction,  guarantees an infinite family of such SEDFs (this also appears in \cite{WeiZha}).

\begin{proposition}
For any prime power $q$ with $q \equiv 1 \mod 4$, there exists a $(q,2,\frac{q-1}{2},\frac{q-1}{4})$-SEDF.
\end{proposition}
\begin{proof}
Let $\alpha$ be a primitive element of $GF(q)^*$, the multiplicative group of the finite field $GF(q)$.   Let $C=\langle \alpha^2 \rangle$, the subgroup of index $2$ in $GF(q)^*$.  Observe that $-1$ is a square in $GF(q)^*$, and $C$ is the set of squares in $GF(q)^*$.  Take $A_1:=C$ and $A_2:=\alpha C$, the coset of $C$ in $GF(q)^*$.  It is known (see, for example, \cite{Wil}) that the cosets of a multiplicative subgroup $H$ of $GF(q)^*$ form a $(q,\frac{q-1}{|H|}, |H|, q-1-|H|)$-EDF in the additive group of $GF(q)$.  Thus $\{A_1,A_2\}$ form a $(q,2,\frac{q-1}{2},\frac{q-1}{2})$-EDF, i.e. the multiset of external differences between $A_1$ and $A_2$, comprises $\frac{q-1}{2}$ occurrences of each non-zero element of $GF(q)$.   In order to be a $(q,2,\frac{q-1}{2},\frac{q-1}{4})$-SEDF, each non-zero element must arise $\frac{q-1}{4}$ times in the multiset $A_1-A_2$ and $\frac{q-1}{4}$ times in the multiset $A_2-A_1$.   Let $x \in GF(q)^*$, and define  
$D_1(x):=\{ (g,h): x=g-h, g \in A_1, h \in A_2 \}$ and $D_2(x):=\{ (g,h): x=g-h, g \in A_2, h \in A_1 \}$.
We exhibit a bijection between $D_1(x)$ and $D_2(x)$: let $(x_1,x_2) \in D_1(x)$.  Then $x=x_1-x_2=(-x_2)-(-x_1)$.  Since $-1$ is a square, we have $A_1=-A_1$ and $A_2=-A_2$, so that $(-x_2,-x_1) \in D_2(x)$ as required.
\end{proof}

This construction can be viewed as a partition of the non-identity group elements into Paley partial difference sets.  

\begin{definition}
A $k$-element subset $D$ of an additive group $\G$ of order $v$ is a $(v,k,\lambda,\mu)$ \emph{partial difference set (PDS)} if the multiset $\mathcal{D}(D)=\{d_1-d_2 \, | \, d_1,d_2 \in D, d_1 \neq d_2 \}$ contains each non-identity element of $D$ exactly $\lambda$ times and each non-identity element of $\G \setminus D$ exactly $\mu$ times. A PDS is called \emph{abelian} if the group $\G$ is abelian.  A PDS $D$ is called \emph{regular} if $D$ does not contain the identity and $D=-D$.  A regular PDS with parameters $(v,\frac{v-1}{2}, \frac{v-5}{4}, \frac{v-1}{4})$, where $v \equiv 1 \mod 4$, is said to be \emph{of Paley type}.
\end{definition}

Further details on Paley PDSs can be found in \cite{Ma}.  The approach of constructing EDFs by partitioning with PDSs is introduced in \cite{DavHucMul}; in particular, Theorem 3.4 of \cite{DavHucMul} establishes that any set of $u$ $(v,k,\lambda,\mu)$ PDSs which partition the non-identity elements of the group $\G$ will form an EDF in $\G$ with parameters  $(ku+1,u, k,ku-1-\lambda-(u-1)\mu)$.  We will show that the PDS approach, applied in the $m=2$ setting using Paley PDSs, will in fact yield SEDFs.

We begin by establishing a useful lemma.
\begin{lemma}\label{D_2Paley}
Let $\G$ be an additive group of order $v$, let $D_1$ be a Paley $(v, \frac{v-1}{2}, \frac{v-5}{4}, \frac{v-1}{4})$ PDS in $\G$, and set $D_2=\G^* \setminus D_1$.  Then $D_2$ is also a Paley PDS with the same parameters as $D_1$.
\end{lemma}
\begin{proof}
Observe that, for any $d\in \G$, the multiset $d-\G$ ($d \in \G$) comprises each element of $\G$ precisely once, while $d-\G^*$ comprises each element of $\G$ except $d$ itself.  Consequently, for any subset $D$ of $\G$, the multiset $D-\G$ comprises $|D|$ copies of each element of $\G$, whereas $D-\G^*$ comprises $|D|$ copies of the elements of $\G\setminus D$ and $|D|-1$ copies of the elements of $D$.   

Each element of $D_1$ occurs $|D_1|-1$ times in the multiset $\G^*-D_1$,  $|D_1|-1$ times in $D_1-\G^*$ and $\frac{v-5}{4}$ times in $D_1-D_1$, hence it occurs in $D_2-D_2$ a total of 
$$(v-2)-(2(|D_1|-1)-\frac{v-5}{4})=\frac{v-1}{4} \mbox{ times}.$$ 
Similar reasoning shows that each element of $D_2$ occurs $\frac{v-5}{4}$ times as an internal difference in the multiset $D_2-D_2$, and hence $D_2$ is a PDS with the given parameters. Clearly $D_2$ is regular: $0 \not \in D_2$ by definition, and $D_2=-D_2$ since $D_1=-D_1$ and $\G^*=-\G^*$.
\end{proof}

\begin{theorem}\label{thm:Paley}
Let $\G$ be an additive abelian group of order $v$, let $D_1$ be a Paley $(v, \frac{v-1}{2}, \frac{v-5}{4}, \frac{v-1}{4})$ PDS in group $\G$ and set $D_2=\G^*\setminus D_1$. Then $\{D_1,D_2\}$ is a $(v, 2, \frac{v-1}{2}, \frac{v-1}{4})$-SEDF.
\end{theorem}
\begin{proof}
By Lemma \ref{D_2Paley}, since $D_1$ is a Paley difference set with parameters $(v, \frac{v-1}{2}, \frac{v-5}{4}, \frac{v-1}{4})$, so is $D_2$.  The fact that $\{D_1,D_2\}$ forms an EDF is a consequence of Theorem 3.4 of \cite{DavHucMul}, and can readily be seen directly, as each element of $\G^*$ occurs $\frac{v-5}{4}+\frac{v-1}{4}=\frac{v-3}{2}$ times in the set of internal differences, and hence $(v-2)-\frac{v-3}{2}=\frac{v-1}{2}$ times in the set of external differences.  To see that this EDF is strong, let $c \in D_1$; then the number of times $c$ occurs in the multiset $D_1-D_2$ is given by $|D_1|-1-\frac{v-5}{4}=\frac{v-1}{4}$.  Now let $d \in D_2$; the number of times $d$ occurs in $D_1-D_2$ is $|D_1|-0-\frac{v-1}{4}=\frac{v-1}{4}$.  Thus $D_1-D_2$ comprises every element of $\G^*$ precisely $\lambda=\frac{v-1}{4}$ times, and reversal yields the same property for $D_2-D_1$.
\end{proof}

\begin{example} 
The following Paley PDS construction is due to \cite{Ma}: let $q$ be an odd prime power, and let $H_1, H_2, \ldots H_{\frac{q+1}{2}}$ be distinct lines in the two-dimensional vector space $\mathbb{F}_q \times \mathbb{F}_q$.  Then $D=\bigcup_{i=1}^{(q+1)/2} (H_i \setminus \{(0,0)\})$ is a Paley PDS in the group $(\mathbb{F}_q^2,+)$, which yields a $(q^2,2, \frac{q^2-1}{2}, \frac{q^2-1}{4})$-SEDF via Theorem \ref{thm:Paley}.

In the above, take $q=3$ and  let $\G=\mathbb{Z}_3 \times \mathbb{Z}_3$.  Then $\{D_1,D_2\}$ where 
$$D_1=\{(0,1), (0,2), (1,0), (2,0)\}$$ 
and 
$$D_2=\{(1,1), (1,2), (2,1), (2,2) \}$$
 is a $(9,2,4,2)$-SEDF - this is the SEDF from Example \ref{lambda=m=2}.
\end{example}

It transpires that Paley PDSs offer, not simply a class of examples, but a characterization of SEDFs of partition type when $m=2$.

\begin{theorem}
Let $\G$ be an additive abelian group of order $v$ and let $D_1, D_2$ be two sets of size $\frac{v-1}{2}$ which partition the non-identity elements of $\G$. Then $\{D_1,D_2\}$ is an SEDF in $\G$ if and only if $D_1$ (and hence $D_2$) is a Paley PDS in $\G$.
\end{theorem}
\begin{proof}
($\Rightarrow$)  Suppose $\{D_1,D_2\}$ is an SEDF.  The condition $\lambda(v-1)=k^2(m-1)$ of equation (\ref{eqn:basic}) with $m=2$ and $k=\frac{v-1}{2}$ yields $\lambda=\frac{v-1}{4}$.  Since $\lambda \in \mathbb{N}$, we must have $v \equiv 1 \mod 4$. Hence a Paley PDS with appropriate parameters is defined for all values of $v$ for which such an SEDF can exist.

We show that if $\{D_1,D_2\}$ is a $(v, 2, \frac{v-1}{2}, \frac{v-1}{4})$-SEDF in $\G$, then $D_1$ is a Paley $(v,\frac{v-1}{2}, \frac{v-5}{4}, \frac{v-1}{4})$ PDS.  Consider the number of times an element $c$ of $D_1$ occurs in the multiset of internal differences of $D_1$.  In $D_1-\G^*$, $c$ occurs $|D_1|-1$ times, while in $D_1-D_2$ it occurs $\frac{v-1}{4}$ times.  Hence in $D_1-D_1$, it occurs $\frac{v-1}{2}-1-\frac{v-1}{4}=\frac{v-5}{4}$ times.  An element $d \in D_2$ occurs $|D_1|$ times in $D_1-\G^*$ and $\frac{v-1}{4}$ times in $D_1-D_2$, i.e. $\frac{v-1}{4}$ times in $D_1-D_1$, as required. 

We must check that $D_1$ is regular.  By definition, $0 \not\in D_1$.  To see that $D_1=-D_1$, we show that if $x$ lies in $D_1$ then so does $-x$.   Note that the elements of $D_1$ are precisely those elements of $\G^*$ which occur $\frac{v-5}{4}$ times as a difference in $\mathcal{D}(D_1)$.  Let $x \in D_1$, and observe that every pair $(a_1,b_1) \in D_1 \times D_1$ such that $x=a_1-b_1$, is in correspondence with the pair $(b_1,a_1) \in D_1 \times D_1$ such that $-x=b_1-a_1$.  Since there are precisely $\frac{v-5}{4}$ pairs, $-x \in D_1$.

Lemma \ref{D_2Paley} now implies that $D_2$ is also a Paley PDS with the same parameters.

($\Leftarrow$) This direction is established in Theorem \ref{thm:Paley}.
\end{proof}

%A table which summarizes the discussion in the above proofs is provided; this schematically represents the subtraction table of $G^*$.

%diagram of subtraction table
%$$
%\begin{array}{|c|c|c|}
%\hline
%- &  D_1& D_2\\ \hline

%D_1& \mbox{Elt of $D_1$:}  \frac{v-5}{4} \mbox{times} &  \mbox{Elt of $D_1$:}  \frac{v-1}{4} \mbox{times} \\
% &  \mbox{Elt of $D_2$:}  \frac{v-1}{4} \mbox{times} &  \mbox{Elt of $D_2$:}  \frac{v-1}{4}  \mbox{times} \\ \hline
%D_2& \mbox{Elt of $D_1$:}  \frac{v-1}{4} \mbox{times} &  \mbox{Elt of $D_1$:}  \frac{v-1}{4} \mbox{times} \\
% &  \mbox{Elt of $D_2$:}  \frac{v-1}{4} \mbox{times} &  \mbox{Elt of $D_2$:}  \frac{v-5}{4} \mbox{times} \\ \hline
%\end{array}
%$$

This characterization tells us that an $(n,2,\frac{n-1}{2},\lambda)$-SEDF can be constructed whenever an abelian Paley PDS of order $n$ can be constructed. For example, constructions are given in \cite{LeuMa} for groups of the form $(\mathbb{Z}_{p^{r_1}})^2 \times (\mathbb{Z}_{p^{r_2}})^2 \times \cdots \times  (\mathbb{Z}_{p^{r_s}})^2$ for $r_1,r_2, \ldots, r_s \in \mathbb{Z}^+$, and in \cite{Pol} for groups of the form $\mathbb{Z}_3^2 \times \mathbb{Z}_p^{4s}$ for $p$ any odd prime. 

We may ask whether there exists an $(n,2,k,\lambda)$-SEDF with $k<\frac{n-1}{2}$.  This is answered in the affirmative in \cite{WeiZha}, where a cyclotomic construction yields SEDFs with parameters $(q,2,\frac{q-1}{4},\frac{q-1}{16})$ and $(q,2,\frac{q-1}{6}, \frac{q-1}{36})$ for prime powers $q$ of certain specific forms.  It is an open question which other parameter sets are possible.

\section{Generalized SEDFs}

In the definition of a strong external difference family, we may relax the condition on uniform set size to obtain the following, introduced in \cite{PatSti}:

\begin{definition}[Generalized Strong External Difference Family] 
Let $\G$ be an additive abelian group of order $n$.  An $(n,m;k_1,\ldots,k_m;\lambda_1,\ldots,\lambda_m)${\bf -generalized strong external difference family} (or $(n,m;k_1,\ldots,k_m;\lambda_1,\ldots,\lambda_m)${\bf -GSEDF}) is a set of $m$ disjoint subsets of $\G$, say $A_1,\ldots,A_m$, such that $|A_i|=k_i$ for $1 \leq i \leq m$ and the following multiset equation holds for every $i$, $1 \leq i \leq m$:
\begin{equation*}
\bigcup_{\{j:j\neq i\}}\D(A_i,A_j)=\lambda_i(\G\setminus\{0\}).
\end{equation*}
\end{definition}

An $(n,m,k,\lambda)$-SEDF is, by definition, an $(n,m;k,\ldots,k; \lambda, \ldots, \lambda)$-GSEDF.

Two examples of GSEDFs were given in \cite{PatSti}:
\begin{itemize}
\item Let $\G=(\mathbb{Z}_n,+)$, $A_1=\{0\}$ and $A_2=\{1,2,\ldots,n-1\}$. This is an $(n,2;1,n-1;1,1)$-GSEDF.

\item Let $\G=(\mathbb{Z}_7,+)$, $A_1=\{1\}$, $A_2=\{2\}$, $A_3=\{4\}$ and $A_4=\{0,3,5,6\}$.  This is a $(7,4;1,1,1,4;1,1,1,2)$-GSEDF.
\end{itemize}
Further new GSEDFs are constructed in \cite{WeiZha}.

The proof strategy of Theorem \ref{thm:dnelambda} may be extended to obtain a necessary condition for the existence of a GSEDF:

\begin{theorem}\label{thm:dneGSEDF}
Suppose $\{A_1, \ldots, A_m \}$ is an $(n,m;k_1,\ldots,k_m;\lambda_1,\ldots,\lambda_m)$-GSEDF, where $m\geq 3$.  Let $\Lambda=\lambda_1 + \cdots + \lambda_m$ and $K=k_1+ \cdots + k_m$.  Then for any $i \in \{1, \ldots, m\}$ for which $k_i>\lambda_i > 1$ and $\lambda_i \leq \frac{\Lambda}{2}$, the following inequality holds:
\begin{align}
\frac{(k_i-1)(\Lambda-2 \lambda_i)}{(K- k_i)(\lambda_i-1)}\leq 1. \label{eq:GSEDF}
\end{align}

\end{theorem}

\begin{proof}
Suppose, without loss of generality, that $i=1$ satisfies the conditions of the theorem statement, i.e. $k_1> \lambda_1 >1$ and $\lambda_1 \leq \frac{\Lambda}{2}$.  We will show that if \eqref{eq:GSEDF} does not hold, then it is possible to find a point $v$ in $A_1$ and $\lambda_1$ internal differences from $v$ which correspond to a point $v^\prime$ in some $A_i$ for $i\neq 1$ and $\lambda_1$ external differences from $v^\prime$, and thereby to construct $\lambda_1+1$ external differences from $A_1$ that are all equal. 

Fix a point $v$ in $A_1$.  Let $I$ be the set of internal differences from $v$
\begin{equation*}
I=\{v-a \, | \, a\in A_1,\ a\neq v\}\subseteq \G\setminus\{0\}.
\end{equation*}
Then $|I|=k_1-1\geq \lambda_1 >0$.

For $x\in A_i$ with $i\neq 1$ let $E_x$ be the set of external differences from $x$ to elements of $A_j$ for any $j\neq 1$:
\begin{equation*}
E_x=\{x-a \, | \, a\in A_j,\ j\neq 1,i\}.
\end{equation*}
Then $|E_x|=K-k_1-k_i$, for any $x \in A_i$.

From the definition of an $(n,m;k_1,\ldots,k_m;\lambda_1,\ldots,\lambda_m)$-GSEDF, each nonzero group element appears $\lambda_j$ times in the multiset of external differences from $A_j$ for each $j=1,2,\dots, m$, and hence $\sum_{j=1}^m \lambda_j= \Lambda$ times in the multiset of all external differences:
\begin{equation}
\{a-b \, | \, a\in A_i,\ b\in A_j,\ i\neq j\}=\Lambda ({\G}\setminus\{0\}). \label{eq:GSexternaldifferences}
\end{equation}  

Furthermore, since the multiset of external differences from $A_1$ comprises $\lambda_1$ copies of each nonzero group element, it is also the case that each nonzero group element occurs precisely $\lambda_1$ times as an external difference from $\cup_{i \neq 1} A_i$ into $A_1$ (since the multiset of such external differences can be obtained by negating the multiset of external differences out of $A_1$). Hence each non-zero group element ocurs $2 \lambda_1$ times as an external difference involving $A_1$.  From this we can deduce that each nonzero group element occurs $\Lambda - 2 \lambda_1$ times as an external difference between sets $A_i$ and $A_j$ with $i\neq j$ and $i,j\neq 1$, so
\begin{equation}
\bigcup_{x \in A_i, i\neq 1} E_x =(\Lambda - 2 \lambda_1)({\G}\setminus \{0\}). \label{eq:2exunion}
\end{equation}
Observe that $\Lambda-2 \lambda_1 \geq 0$ is guaranteed by the initial conditions.

Suppose we could find an element $v^\prime \in A_i$ for some $i\neq 1$ for which $|E_{v^\prime} \cap I|\geq \lambda_1$.

%\begin{figure}
%\caption{Write a caqption here!}
%\begin{tikzpicture}[scale=1]
%\draw[step=1cm,gray,very thin] (0,0) grid (10,10);
%The hyperplane 
%\draw (5,10) ellipse (2cm and .5cm);
%\draw (5,8.5) ellipse (2cm and .5cm);
%\draw (5,6.5) ellipse (2cm and .5cm);
%\draw (5,4.5) ellipse (2cm and .5cm);
%\draw (5,2) ellipse (2cm and .5cm);

%\fill[black] (4.5,10)circle (0.05cm);
%\fill[black] (4,9.7)circle (0.05cm);
%\fill[black] (5.5,9.8)circle (0.05cm);

%\node at (4.5,10.5) {$v$};

%The lines in a parallel class
%\draw  (1,2) -- (5,8);
%\draw  (3,2) -- (5,8);
%\draw  (4,2) -- (5,8);
%\draw  (5,2) -- (5,8);
%\draw  (7,2) -- (5,8);
%\draw  (8,2) -- (5,8);
%\draw  (9,2) -- (5,8);

%Labels
%\node at (1.6,3.5) {$\infty$};
%\node at (5,8.4) {$P_{\infty}$};
%\node at (5.8,3.5) {$\ldots$};
%\node at (8,9) {${\infty}$};

%\node at (1,1.5) {$L_0$};
%\node at (4,1.5) {$ \underbrace{\hspace{2cm}} $};
%\node at (4,1) {$1$};
%\node at (8,1.5) {$ \underbrace{\hspace{2cm}} $};
%\node at (8,1) {${\frac{q^{n-1}-1}{q-1}}$};

%\end{tikzpicture}
%\end{figure}

 Let $\delta_1, \ldots \delta_{\lambda_1}$ be distinct elements of $E_{v^\prime}\cap I$.  Let $u_t=v-\delta_t$ for $1 \leq t \leq \lambda_1$.  Then $u_1, \ldots, u_{\lambda_1}$ are distinct elements of $A_1$, as $\delta_1, \ldots, \delta_{\lambda_1}$ are distinct elements of $I$.  Let ${u_t}^\prime=v^\prime-\delta_t$ for $1 \leq t \leq \lambda_1$.  Then ${u_1}^\prime, \ldots, {u_{\lambda_1}}^\prime$ are distinct elements of $\cup_{j\neq 1,i} A_j$ as  $\delta_1, \ldots \delta_{\lambda_1}$  are distinct elements of $E_{v^\prime}$.  (We note that the ${u_t}^\prime$ may lie in different $A_j$ and $A_\ell$, or they may all occur in a single $A_j$ but that does not affect the rest of this argument.)  Let $v-v^\prime=\gamma\in{\G}\setminus\{0\}$. We observe that, for each $1 \leq t \leq \lambda_1$,  
\begin{align*}
{u_t}-{u_t}^\prime &=(v-\delta_t)-(v^\prime-\delta_t),\\
&=v-v^\prime,\\
&=\gamma.
\end{align*}
So $\gamma \in \G \setminus \{0\}$ occurs as an external difference from $A_1$ a total of $\lambda_1+1$ times  - a contradiction.  Thus for all $v^\prime \in A_i$ ($i\neq 1$), we must have $|E_{v^\prime} \cap I|\leq \lambda_1 -1$.

We now count the number $N$ of pairs $(\theta,E_x)$ where $\theta\in I\cap E_x$ and $x\in A_i$ for some $i\neq 1$. There are $|I|=k_1-1$ choices for $\theta$.  As each nonzero element of $\G$ occurs $\Lambda - 2 \lambda_1$ times in $\bigcup_{x \in A_i, i\neq 1} E_x $, we see that for each of these $\theta$ there are $\Lambda - 2 \lambda_1$ values of $x$ for which $\theta\in E_x$, so that $N=(k_1-1)( \Lambda - 2 \lambda_1)$.

The number of distinct sets $E_x$ with $x\in A_i$ for some $i\neq 1$ is $\sum_{j=2}^m k_j=K-k_1$.  By the Pigeonhole Principle there exists $x$ for which the set $E_x$ contains at least
\begin{equation*}
\frac{N}{K-k_1}=\frac{(k_1-1)(\Lambda - 2 \lambda_1)}{K-k_1}
\end{equation*}
elements of $I$. Whenever this quantity is at least $\lambda_1$, i.e. strictly greater than $(\lambda_1-1)$, we have $|E_{v^\prime} \cap I|\geq \lambda_1$ for some $v^\prime$.  Thus no such GSEDF will exist if
\begin{equation*}
\frac{(k_1-1)(\Lambda- 2 \lambda_1)}{K- k_1} > \lambda_1 -1.
\end{equation*}
It is clear that the same argument will hold with $1$ replaced by any appropriate $i \in \{1, \ldots, m\}$ for which the conditions in the theorem statement are satisfied.
\end{proof}

Observe that taking $k_i=k$ and $\lambda_i=\lambda$ for all $1 \leq i \leq m$ yields Theorem \ref{thm:dnelambda}, while further setting $k_i=k$ and $\lambda_i=2$ for all $1 \leq i \leq m$ yields Theorem \ref{thm:dne}.

\section{Concluding remarks}
This paper establishes various conditions under which SEDFs and GSEDFs can exist, and establishes a structural characterization of partition type SEDFs in the $m=2$ case.  Future directions are two-fold; the fine-tuning of such necessary conditions, with the aim of completely characterising the possible parameters sets, and the development of further construction methods and structural characterizations. 
One specific open problem is whether SEDFs with $\lambda=m=2$ exist when $k$ is less than the maximum possible size of $\frac{n-1}{2}$.  It would be desirable to gain further understanding of the general case when $m=2$. By \cite{MarSti}, the next-smallest $m$ for which SEDFs may exist is $m=5$; further investigation of this case would be another natural focus.

\subsection*{Acknowledgements} Thanks to Siaw-Lynn Ng for helpful discussions.
%\bibliography{sedfbib}

\end{document}